\documentclass[a4paper,12pt,oneside,reqno]{amsart}

\usepackage{hyperref}
\usepackage{amsfonts,amssymb,amsmath,amsthm}

\usepackage[utf8]{inputenc}
\usepackage{bbm}
\usepackage[english]{babel}
\usepackage[rightcaption]{sidecap}
\sidecaptionvpos{figure}{c}

\usepackage{graphicx, psfrag}

\usepackage{subfigure}

\usepackage[headinclude,DIV13]{typearea}
\areaset{15.1cm}{25.0cm}
\newtheorem{theorem}{Theorem}[section]
\newtheorem{lemma}[theorem]{Lemma}

\theoremstyle{remark}
\newtheorem*{remark}{Remark}

\def\paragraph#1{\noindent \textbf{#1}}

\numberwithin{equation}{section}

\def\ddd{\mathrm{d}}
\def\eee{\mathrm{e}}
\def\dd{\mathtt{d}}

\def\<{\langle}
\def\>{\rangle}

\def\b{\beta}
\def\e{\epsilon}

\def\g{\gamma}

\def\l{\lambda}

\def\s{\sigma}
\def\t{\tau}

\def\R{{\Bbb R}}
\def\N{{\Bbb N}}
\def\P{{\Bbb P}}
\def\Z{{\Bbb Z}}

\def\C{{\Bbb C}}
\def\E{{\Bbb E}}
\def\M{{\Bbb M}}

\let\cal=\mathcal

\def\FF{{\cal F}}

\def\MM{{\cal M}}

\def\XX{{\cal X}}

\def \b {{\beta}}
\def \s {{\sigma}}

\def \t {{\tau}}

\def \g {{\gamma}}
\def \l {{\lambda}}

\def \ba {\begin{array}}
\def \ea {\end{array}}

\newcommand{\be}{\begin{equation}}
\newcommand{\ee}{\end{equation}}

\newcommand{\bea}{\begin{eqnarray}}
\newcommand{\eea}{\end{eqnarray}}
\def\TH(#1){\label{#1}}\def\thv(#1){\ref{#1}}
\def\Eq(#1){\label{#1}}\def\eqv(#1){(\ref{#1})}

\def\cov{\hbox{\rm Cov}}

\def \1{\mathbbm{1}}

\DeclareMathOperator{\length}{length}
\DeclareMathOperator*{\wlim}{wlim}

\def\eee{\hbox{\rm e}}

\allowdisplaybreaks

\begin{document}

\title[The Phase diagram of
the complex BBM
energy model]{The Phase diagram of\\
the complex branching Brownian motion\\
energy model} 

\author[L.~Hartung]{Lisa Hartung}
\address{L.~Hartung\\Courant Institute of Mathematical Sciences\\
New York University\\ 251 Mercer Street\\ 10012 New York, NY, USA}
\email{hartung@cims.nyu.edu}

\author[A.~Klimovsky]{Anton Klimovsky}
\address{A.~Klimovsky\\Fakult\"at f\"ur Mathematik\\Universit\"at Duisburg-Essen\\Thea-Leymann-Str. 9\\
45127 Essen, Germany}
\email{anton.klymovskiy@uni-due.de }



\date{\today}

\begin{abstract}
We complete the analysis of the phase diagram of the complex branching Brownian
motion energy model by studying Phases I, III and boundaries between all three
phases (I-III) of this model. For the properly rescaled partition function, in
Phase III and on the boundaries I/III and II/III, we prove a central limit
theorem with a random variance. In Phase I and on the boundary I/II, we prove an
a.s.\ and $L^1$ martingale convergence. All results are shown for any given
correlation between the real and imaginary parts of the random energy.
\end{abstract}

\thanks{Part of this work was done while L.H.~was supported by the German
Research Foundation in the Bonn International Graduate School in Mathematics
(BIGS), and the Collaborative Research Center 1060 ``The Mathematics of Emergent
Effects''. The authors thank the University of Bonn and the University of
Duisburg-Essen for hospitality. }

\maketitle

\section{Introduction}
\label{sec:intro}

Random energy models (REM) suggested by Derrida~\cite{Derrida_REM3,Derrida_GREM}
turned out to be a useful and instructive ``playground'' in the studies of
strongly correlated random systems on large/high-dimensional state spaces, see,
e.g., the recent reviews \cite{PanchenkoBook2013,Kistler2015,Bovier2016}. In
this context, \textit{branching Brownian motion} (BBM) viewed as a random energy
model plays a special rôle. It turns out that BBM has correlations which are
exactly at the borderline between the regime of weak correlations (REM
universality class\footnote{= the same phase diagram as for the field of
independent random energies.}) and the one of strong correlations\footnote{=
different phase diagram comparing to the REM one, due to the strictly larger
leading order of the minimal energy than the one for the independent field of
random energies.}. Apart from that, BBM is a particularly transparent
representative for a whole class of models with similar (so-called logarithmic)
correlation strength: Gaussian free
field~\cite{ZeitouniLN,BisLou13,BiskupLouidor2016full}; Gaussian multiplicative
chaos/cascades~\cite{RhodesVargas2014ChaosReview,BarralJinMandelbrot2010};
characteristic polynomials of random matrices and number-theoretic
models~\cite{FyodorovHiaryKeating2012,ArguinBeliusBourgade2015,ArguinBeliusHarper2015}, 
cover times~\cite{BeliusKistler2014subleading}, etc.

In this paper, we focus on the \textit{complex-valued BBM energy model} and show
that this model lies exactly at the borderline of the complex REM universality
class. This means that the \textit{phase diagram} of the model is the same as in
the complex REM, cf.~Derrida~\cite{Derrida_zeros} and \cite{KaKli14}. However,
the fluctuations of the partition function of this model are already influenced
by the strong correlations and differ from those of the REM in \textit{all
phases} of the model, as we show in this work (and in~\cite{HK15}).

The motivation to consider the complex-valued setup is two-fold:
\begin{enumerate}

\item \textbf{Critical phenomena.} Lee and Yang~\cite{LY52} observed that
\textit{phase transitions} (= analyticity breaking of the log-partition
function) occur at critical points due to the accumulation of \textit{complex
zeros} of the partition function (viewed as a function of the temperature)
around the critical points on the real line, as the size of the system tends to
infinity (= thermodynamic limit).

\item \textbf{Quantum physics and interference phenomena.} The formalism of quantum physics is
based on the sums (and integrals) of \textit{complex exponentials} which
naturally leads to cancellations between the magnitudes of the summands in the
partition function. This is a manifestation of the interference phenomenon, see,
e.g., \cite{Dobrinevski2011interference}.

\end{enumerate}

\subsection{Branching Brownian motion.} 

Before stating our results, let us briefly recall the construction of a BBM.
Consider a canonical continuous branching process: a \textit{continuous time
Galton-Watson} (GW) process \cite{AN}. It starts with a single particle located
at the origin at time zero. After an exponential time of parameter one, this
particle splits into $k \in \Z_+$ particles according to some probability
distribution $(p_k)_{k \geq 0}$ on $\Z_+$. Then, each of the new-born particles
splits independently at independent exponential (parameter $1$) times again
according to the same $(p_k)_{k \geq 0}$, and so on. We assume that
$\sum_{k=1}^\infty p_k=1.$\footnote{This implies that $p_0 = 0$, so none of the
particles ever dies.} In addition, we assume that $\sum_{k=1}^\infty k p_k=2$
(i.e., the expected number of children per particle equals two)
. Finally, we assume that $K := \sum_{k=1}^\infty k(k-1)p_k<\infty$ (finite
second moment). At time $t = 0$, the GW process starts with just one particle.

For given $t \geq 0$, we label the particles of the process as $i_1(t),\dots,
i_{n(t)}(t)$, where $n(t)$ is the total number of particles at time $t$. Note
that under the above assumptions, we have $\E\left[ n(t) \right]=\eee^t$. For $s
\leq t$, we denote by $i_k(s,t)$ the unique ancestor of particle $i_k(t)$ at
time $s$. In general, there will be several indices $k, l$ such that $i_k(s,t) =
i_l(s,t)$. For $s, r \leq t$, define the time of the most recent common ancestor
of particles $i_k(r,t)$ and $i_l(s,t)$ as
\begin{align}
\label{eq:intro:gw-overlap}
d(i_k(r,t), i_l(s,t))
:= \sup \{u \leq s \wedge r \colon i_k(u,t) = i_l(u,t)\}.
\end{align} 
For $t \geq 0$, the collection of all ancestors naturally induces the random tree
\begin{align}
\label{eq:gw-tree}
\mathbb{T}_t := \{i_k(s,t) \colon 0 \leq s \leq t, 1\leq k \leq n(t) \}
\end{align}
called the \textit{GW tree up to time $t$}. We denote by $\mathcal{F}^{\mathbb{T}_t}$
the $\s$-algebra  generated by the GW process up to time $t$.

In addition to the genealogical structure, the particles get a
\textit{position} in $\R$. Specifically, the first particle starts at the origin
at time zero and performs Brownian motion until the first time when the GW
process branches. After branching, each new-born particle independently performs
Brownian motion (started at the branching location) until their respective next
branching times, and so on. We denote the positions of the $n(t)$ particles at
time $t \geq 0$ by $x_1(t),\dots, x_{n(t)}(t)$.

We define BBM as a family of Gaussian processes, 
\begin{align}
\label{eq:bbm-process}
x_t := \{ x_1(s,t),\dots, x_{n(t)}(s,t) \colon s \leq t \}
\end{align}
indexed by time horizon $t \geq 0$. Note that conditionally on the underlying GW tree these Gaussian processes have
the following covariance
\begin{align}
\E \left[x_k(s,t) x_l(r,t) \mid \mathcal{F}^{\mathbb{T}_t} \right] = d(i_k(s,t), i_l(r,t)), \quad  s, r \in [0,t], \quad k,l \leq n(t)
.
\end{align}

In what follows, to lighten the notation, we will simply write $x_i(s) :=
x_i(s,t)$, $i \leq n(t)$, $s \leq t $ hoping that this will not cause confusion
about the parameter $t \geq 0$.

\subsection{A model of complex-valued random energies}

In this section, we introduce the \textit{complex BBM random energy model}.

Let  $\rho \in [-1,1]$. For any $t \in \R_+$, let $X(t) := (x_k(t))_{k\leq
n(t)}$ and $Y(t) :=(y_k(t))_{k\leq n(t)} $ be two BBMs with the same underlying
GW tree such that, for $k \leq n(t)$,
\begin{align}
\cov(x_k(t),y_k(t))=|\rho| t
.
\end{align}
Note that
\begin{align}\Eq(cor.1)
Y(t)\overset{\mathrm{D}}{=} \rho X(t)  +\sqrt{1-\rho^2} Z(t)
,
\end{align}
where ``$\overset{\mathrm{D}}{=}$'' denotes equality in distribution and $Z(t)
:= (z_i(t))_{i\leq n(t)}$ is a branching Brownian motion independent from $X(t)$
and with the same underlying GW process. Representation $\eqv(cor.1)$ allows us
to handle arbitrary correlations by decomposing the process $Y$ into a part
independent from $X$ and a fully correlated one.

We define the \textit{partition function} for the complex
BBM energy model with correlation $\rho$ at inverse temperature $\b := \s+i\t\in\C $ by
\begin{align}
\Eq(real.1)
\XX_{\b,\rho}(t) := \sum_{k=1}^{n(t)}\eee^{\s x_k(t)+i\t y_k(t)}
.
\end{align}

\subsection{Notation.} 

By $\mathcal{L}[\cdot]$, $\mathcal{L}[\cdot \mid \cdot ]$, and $\Longrightarrow$ or $\wlim$, we denote the law, conditional law, and weak convergence respectively.

\subsection{Main results} 

\begin{figure}[h]
\begin{minipage}[b]{0.59\textwidth}
\centering
\includegraphics[width=\textwidth]{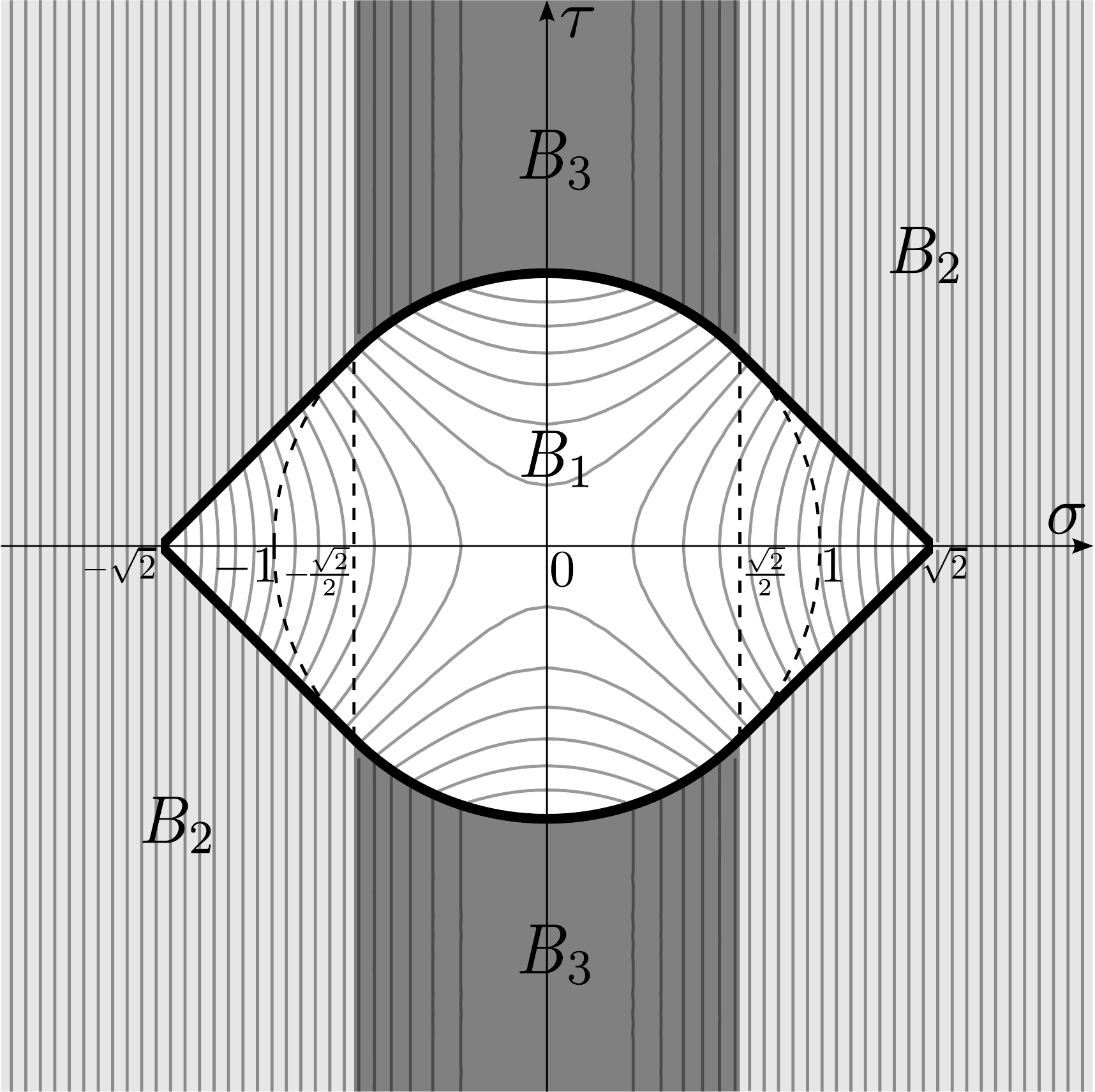}
\end{minipage}
\hfill
\centering

\caption{\small Phase diagram of the REM and the BBM energy model.
The grey curves are the level lines of the limiting log-partition function,
cf.~\eqref{eq:limiting-log-partition-function}. 
This paper deals with phases $B_1$ and $B_3$ and the boundaries. For a treatment of phase $B_2$, see \cite{HK15}.}
\vfill
 \label{fig-rem-phase-diagram}
\end{figure}

Let us specify the three domains depicted on Figure~\ref{fig-rem-phase-diagram} analytically:
\begin{equation}
\begin{aligned}
B_{1}
:=
\C \setminus \overline{B_2\cup B_3},
\quad
&
B_{2}
:=
\{ \sigma + i \tau \in \C \colon 2\sigma^2 >  1, |\sigma|+|\tau| > \sqrt {2}\}
,
\\
&
B_{3}
:=
\{ \sigma + i \tau  \in \C \colon 2\sigma^2 < 1, \sigma^2+\tau^2> 1\}.
\end{aligned}
\end{equation}

\begin{remark}
Some of our results will be stated under the \textit{binary branching}
assumption (i.e., $p_k = 0$ for all $k > 2$). Existence of all moments for the number of children of a given particle would also suffice for all our proofs and will not require
essential changes.
\end{remark}

Our first result states that the complex BBM energy model indeed has the phase diagram depicted on Figure~\ref{fig-rem-phase-diagram}.

\begin{theorem}[Phase diagram]
\label{Cor:phase-diagram}
For any $\rho \in [-1,1]$, and any $\beta \in \C $, the complex BBM energy model with binary branching has the same log-partition function
and the phase diagram (cf., \textup{Figure~\ref{fig-rem-phase-diagram}}) as the
complex REM, i.e.,
\begin{align}
\label{eq:limiting-log-partition-function}
\lim_{t \uparrow \infty} \frac{1}{t} \log \XX_{\b,\rho}(t) =
\begin{cases}
1 + \frac{1}{2}(\sigma^2 - \tau^2), & \beta \in \overline{B_1},
\\
\sqrt{2}|\sigma|, & \beta \in \overline{B_2},
\\
\frac{1}{2}+\sigma^2, & \beta \in \overline{B_3}
\end{cases}
\end{align}
in probability.
\end{theorem}
See Section~\ref{sec:phase-diagram} for a proof.
\begin{remark}
It is conjectured that the convergence in
\ref{eq:limiting-log-partition-function} also holds in $L^1$, see
\cite[Theorem~2.15]{KaKli14} for a related result for the REM.
\end{remark}

\subsection{A class of martingales} 

In the centre of our analysis are the following martingales
\begin{equation}\Eq(mart.1)
\MM_{\s,\t}(t)=\eee^{-t\left(1+\frac{\s^2-\t^2}{2}\right)}\XX_{\b,\rho}(t)
=\sum_{k=1}^{n(t)} \eee^{-t\left(1+\frac{\s^2-\t^2}{2}\right)}\eee^{\s x_k(t)+i\t y_k(t)}.
\end{equation}

Note that, for $\b= \s \in [0, \frac{1}{\sqrt{2}})$, $\MM_{\s,0}(t)$ coincides
with the McKean martingale introduced in \cite{BovHar13}, where it was proven
that these martingales converge almost surely and in $L^1$ to a non-degenerate
limit. 

The next theorem states that for $\b\in B_1$ the martingales $\MM_{\s,\t}(t)$
are in $L^p$ for some $p>1$. For $\vert \b\vert <1$, this was already proven in
\cite[Proposition~A.1]{HK15}.

\begin{theorem}[$L^p$ martingale convergence in $B_1$]\TH(TH.B1)
For $\b=\s+i\t$ with $\b\in B_1, |\beta| \geq 1$, and any $\rho\in[-1,1]$,
$\mathcal{M}_{\s,\t}(t)$  is a martingale with expectation $1$ and it is in
$L^p$ for $p\leq \frac{\sqrt{2}}{\s}$. Hence, the limit
\be\Eq(B1.1) \lim_{t\uparrow
\infty}\mathcal{M}_{\s,\t}(t) =: \mathcal{M}_{\s,\t} 
\ee 
exists a.s., in
$L^1$, and is non-degenerate. 
\end{theorem}
See Section~\ref{sec:b1} for a proof.
\begin{remark} For $\b\in B_1, |\beta| < 1$, and any $\rho\in[-1,1]$, it has been proven in \cite[Proposition~A.1]{HK15} that $\mathcal{M}_{\s,\t}(t)$ is $L^2$-bounded.
\end{remark}
On the boundary $B_{1,2}$ between phases $B_1$ and $B_2$, i.e., on the set
\be\Eq(bound.1)
B_{1,2} := \overline{B_1} \cap \overline{B_2}  = \{\s+i\t\in\C \colon \vert \s\vert>1/\sqrt{2}, \vert \s\vert+\vert \t\vert=\sqrt{2} \}
\ee 
a similar result still holds.

\begin{theorem}[$L^p$ martingale convergence on $B_{1,2}$]\TH(TH.B1.2)
For $\b \in B_{1,2}$ and any $\rho\in[-1,1]$, we have that
$\mathcal{M}_\b(t)$ is a  $L^p$-bounded martingale, for any $p < \sqrt{2}/\s$ with expectation $1$. Hence, the limit
\be\Eq(B1.2)
\lim_{t\uparrow \infty}\mathcal{M}_{\s,\t}(t) =: \mathcal{M}_{\s,\t}
\ee
exists a.s.\ in $L^1$, and is non-degenerate.

\end{theorem}
See Section~\ref{sec:B12} for a proof.

Note that the martingales $\MM_{\s,\t}(t)$ satisfy a recursive equation of the form
\be\Eq(smooth.1)
\mathcal{L}\left[
\MM_{\s,\t}(t+r)
\right]
=
\mathcal{L}
\Big[
\sum_{k=1}^{n(r)} a_k(r)\MM_{\s,\t}^{(k)}(t )
\Big]
,
\ee
where $\MM_{\s,\t}^{(k)}(t-r)$ are i.i.d.\ copies of $\MM_{\s,\t} (t )$ and
$a_k(r)\in \C$ are some complex weights independent from
$\MM_{\s,\t}^{(k)}(t-r), k\in\N$. If a limit $\mathcal{M}_{\s,\t}$ as $t\uparrow
\infty$ of $\MM_{\s,\t}(t+r)$ exists, then it would have  to satisfy the
equation
\be\Eq(smooth.2)
\mathcal{L}\left[
\MM_{\s,\t}
\right]
=
\mathcal{L}
\Big[
\sum_{k=1}^{n(r)} a_k(r)\MM_{\s,\t}^{(k)}
\Big]
,
\ee
where $\MM_{\s,\t}^{(k)}$ are i.i.d.\ copies of $\MM_{\s,\t}$. This type of
equation is called {\it complex smoothing transform}. A detailed study on how
solutions to such equations with complex weights look like was recently done by
Meiners and Mentemeier~\cite{MeMe16}, see also the recent paper by Kolesko and Meiners~\cite{KoleskoMeiners2016}. The case of real-valued scalar weights was treated by Alsmeier and Meiners \cite{AM13} and by Iksanov and Meiners~\cite{IM15}.

The following three results cover the strip $\vert \sigma \vert<1/\sqrt{2}$ and
basically are ``central limit theorems'' (CLTs) with random variance. 
\begin{theorem}[CLT with random variance for $\vert \sigma \vert<1/\sqrt{2}$]\TH(CLT.1)
Let $\beta=\s+i\t$ with $|\s|<1/\sqrt{2}$ and $\rho\in[-1,1]$.
For $\b \in B_1$, 
\begin{equation}\Eq(clt.1)
\wlim_{r \uparrow\infty} \wlim_{t \uparrow\infty} \mathcal{L} \left[
\frac{\MM_{\s,\t}(t+r)-\MM_{\s,\t}(r)}{ \eee^{r(1-\s^2-\t^2)}} ~\Big\vert~ \mathcal{M}_{2\s,0}
\right]
=
\mathcal{N}\left(0,C_1\MM_{2\s,0}\right)
,
\end{equation}
where  $C_1 > 0$ is some constant. 
\end{theorem}
See Section~\ref{sec:b1} for a proof.

\begin{remark}
A result resembling Theorem~\ref{CLT.1} (i) was obtained by Iksanov and Kabluchko in \cite{IK15} for $\beta \in \R$.
\end{remark}

\begin{remark}
The
appearance of the random variance in Theorem~\ref{CLT.1} (and in the following ones) is in sharp contrast with the
REM~\cite{KaKli14} and generalized REM~\cite{KaKli14_G}, where CLTs with
deterministic variance hold for $\beta$ in the strip $\vert \sigma
\vert<1/\sqrt{2}$.
\end{remark}

\begin{theorem}[CLT with random variance in $B_3$]\TH(clt.B3)
For $\b \in B_3$, $\rho\in[-1,1]$ and binary branching,
\begin{equation}\Eq(clt.1.1)
\mathcal{L}
\left[
\frac{\XX_{\b,\rho}(t)}{ \eee^{t(1/2+\s^2)}} ~\Big\vert~ \MM_{2\s,0}
\right]
\underset{t \uparrow\infty}{\Longrightarrow} \mathcal{N}\left(0,C_2\MM_{2\s,0}\right)
,
\end{equation}
where $C_2 > 0$ is some constant.
 
\end{theorem}
See Section~\ref{sec:clt.B3} for a proof.

A similar result also holds on the boundary between phases $B_1$ and $B_3$, i.e., on the set
\be\Eq(bound.2)
B_{1,3} := \overline{B_1} \cap \overline{B_3} = \{\s+i\t\in\C \colon  \s^2+  \t^2=1, \vert \s\vert < 1 / \sqrt{2}\}
.
\ee

\begin{theorem}[CLT with random variance on $B_{1,3}$]\TH(CLT.2)
For $\b \in B_{1,3}$, $\rho\in [-1,1]$, and binary branching,
\begin{equation}\Eq(clt.2.1)
\mathcal{L} \left[
\frac{\XX_{\b,\rho}(t)}{ \sqrt{t} \eee^{t(1/2+\s^2)}} ~\Big\vert~ \MM_{2\s,0}
\right]
 \underset{t \uparrow \infty}{\Longrightarrow} \mathcal{N}\left(0,C_3\MM_{2\s,0}\right)
,
\end{equation}
where $C_3 > 0$ is some constant.
\end{theorem}
See Section~\ref{sec:B13} for a proof.

Recall that the behaviour of the partition function at $\beta = \sqrt{2}$ is determined by the martingale $\MM_{1,0}(t)$, which is related to another martingale -- the so-called \textit{derivative martingale} $\mathcal{Z}(t)$:
\be \Eq(real.2)
\mathcal{Z}(t):=\sum_{i=1}^{n(t)}(\sqrt{2}t-x_k(t))\eee^{-\sqrt{2}(\sqrt{2} t-x_k(t))}.
\ee
Lalley and Sellke proved in \cite{LS} that $\mathcal{Z}(t)$ converges a.s.\ as $t \to \infty $ to a non-trivial limit $\mathcal{Z}$ which is a positive and a.s.\ finite random variable.

At the boundary,
\be\Eq(bound.3)
B_{2,3} := \overline{B_2} \cap  \overline{B_3} = \Big\{\s+i\t\in\C \colon   \vert \s\vert =\frac{1}{\sqrt{2}}, \vert\t\vert \geq \frac{1}{\sqrt{2}}\Big\},
\ee
including the \textit{triple point} 
\begin{align}
\label{}
\beta_{1,2,3} := \overline{B_1} \cap  \overline{B_2} \cap \overline{B_3} = (1+i)/\sqrt{2}
,
\end{align}
after appropriate rescaling, we have the following CLT with random variance.

\begin{theorem}[CLT with random variance for $|\s|=1/\sqrt{2}$]\TH(CLT.B)
Let $\beta = \s+i\t$ with $|\s| = 1/\sqrt{2}$ and $\rho \in [-1,1]$ and assume binary branching. Then:
\begin{itemize}  \item[(i)] For $\t > 1/\sqrt{2}$, 
\begin{equation}\Eq(clt.b.1)
\wlim_{r \uparrow\infty} \wlim_{t \uparrow\infty} \mathcal{L}\left[ r^{1/4} \cdot \frac{\XX_{\b,\rho}(t+r)}{ \eee^{(t+r)(1/2+\s^2)}} ~\Big\vert~ \mathcal{F}_r\right] 
= \mathcal{N}\left(0,C_2 \sqrt{\frac{2}{\pi}}\mathcal{Z}\right)
.
\end{equation}

\item[(ii)] For $\t = 1/\sqrt{2}$,
\begin{equation}\Eq(clt.b.2)
\wlim_{r \uparrow\infty} \wlim_{t \uparrow\infty} \mathcal{L} \left[  \frac{r^{1/4}}{\sqrt{t}}\cdot \frac{\XX_{\b,\rho}(t+r)}{ \eee^{(t+r)(1/2+\s^2)}} ~\Big\vert~ \mathcal{F}_r\right] = \mathcal{N}\left(0,C_3\sqrt{\frac{2}{\pi}}\mathcal{Z}\right)
.
\end{equation}

\end{itemize}
\end{theorem}
See Section~\ref{sec:B_2-B_3-boundary} for a proof.

\subsection{Organization of the rest of the paper.} The remainder of the paper
is organized as follows. In Section~\ref{sec:b1}, we prove Theorems~\thv(TH.B1)
and \thv(CLT.1) concerning the behaviour of the partition function in Phase
$B_1$. In Section~\ref{sec:b3}, we treat Phase $B_3$. We start with a second
moment computation which is then in the next subsection generalised to a
constrained higher moment computation. Finally, in Section~\ref{sec:clt.B3}, we
prove Theorem~\thv(clt.B3). The boundaries $B_{1,3}, B_{2,3}$ (Theorems~\ref{CLT.2} and
\ref{CLT.B}) are proved in Section~\ref{sec:boundaries}.
Section~\ref{sec:phase-diagram} contains the proof of
Theorem~\ref{Cor:phase-diagram}.

\section{Proof of results for phase $B_1$} 
\label{sec:b1}

We start by proving the martingale convergence of $\mathcal{M}_{\s,\t}(t)$.

\begin{proof}[Proof of Theorem \thv(TH.B1)]
One readily checks that $\mathcal{M}_{\s,\t}(t)$ is a martingale with expectation $1$. Next, we compute the $\frac{\sqrt{2}}{\s}$-moment of $\mathcal{M}_\b(t)$. To do this, first consider
\be\Eq(m.1)
\E\Big[\Big\vert \sum_{k=1}^{n(t)} \eee^{\s x_k(t)+i\t y_k(t)} \Big\vert^{ \frac{\sqrt{2}}{\s}}\Big]
=
\E\Big[\Big\vert \sum_{k=1}^{n(t)} \eee^{(\s+i\rho \t ) x_k(t)+i\sqrt{1-\rho^2}\t z_k(t)} \Big\vert^{ \frac{\sqrt{2}}{\s}}\Big],
\ee
where we used  Representation \eqv(cor.1). The right-hand side of \eqv(m.1) is equal to
\begin{multline}
\Eq(m.2)
\E\Big[\Big(\Big\vert \sum_{k=1}^{n(t)} \eee^{(\s+i\rho \t ) x_k(t)+i\sqrt{1-\rho^2}\t z_k(t)} \Big\vert^2\Big)^{ \frac{1}{\sqrt{2}\s}}\Big]
\\
= 
\E\Big[\Big(  \sum_{k,j=1}^{n(t)} \eee^{ \s   (x_k(t)+x_j(t))+ i\rho \t(x_k(t)-x_j(t))+i\sqrt{1-\rho^2}\t( z_k(t)-z_j(t))} \Big)^{ \frac{1}{\sqrt{2}\s}}\Big].
\end{multline}
By Jensen's inequality for the conditional expectations, and because $1/\sqrt{2}\s <1$, for $\s>1/\sqrt{2} $, \eqv(m.2) is bounded from above by 
\begin{multline}
\Eq(m.3)
\E\Big[\Big(  \sum_{k,j=1}^{n(t)} \eee^{ \s   (x_k(t)+x_j(t))+ i\rho \t(x_k(t)-x_j(t))}\E \Big[ \eee^{\sqrt{1-\rho^2}\t( z_k(t)-z_j(t))}\Big] \Big)^{ \frac{1}{\sqrt{2}\s}}\Big]
\\
= \E\Big[\Big(\eee^{-  (1-\rho^2)\t(t-q_{k,j}) }  \sum_{k,j=1}^{n(t)} \eee^{\s   (x_k(t)+x_j(t))+ i\rho \t(x_k(t)-x_j(t))}  \Big)^{ \frac{1}{\sqrt{2}\s}}\Big]
.
\end{multline}
We set
\be\Eq(m.4)
q_{k,j} := d(x_k(t),x_j(t)).
\ee
By the branching property,
\be\Eq(m.3.1)
x_k(t)-x_j(t) \stackrel{D}{=} x_{k'}^{(1)}(t-q_{k,j})-x_{j'}^{(2)}(t-q_{k,j}),
\ee
 where $k'$ and $j'$ are two BBM particles at time $t-q_{k,j}$ from two independent copies $X^{(1)}(\cdot)$ and $X^{(2)}(\cdot)$ of a BBM and let $n^{(1)}(s)$ and  $n^{(2)}(s)$ denote the number of particles of $X^{(1)}$, resp.\ $X^{(2)}$, at time $s$. Using  \eqv(m.3.1), we rewrite \eqv(m.3) as  
\be\Eq(m.5) 
\E\Big[\Big(\eee^{-  (1-\rho^2)\t^2(t-q_{k,j}) }  \sum_{k=1}^{n(q_{k,l})} \eee^{2 \s   x_k(q_{k,l})} \sum_{\substack{k'\leq n^{(1)}(t-q_{k,j}),\\j'\leq n^{(2)}(t-q_{k,j})}} \eee^{i\rho \t\left(x_{k'}^{(1)}(t-q_{k,j})-x_{j'}^{(2)}(t-q_{k,j})\right)} 
\Big)^{\frac{1}{\sqrt{2}\s}}\Big]
.
\ee
Using again Jensen's inequality for the conditional expectation ($\sum (\ldots)^{\sqrt{2}\s}\leq (\sum (\ldots))^{\sqrt{2}\s} $ since $\sqrt{2}\s >1$), we bound \eqv(m.5) from above by
 \be\Eq(m.6) 
 \E\Big[\eee^{-  \frac{(1-\rho^2)\t^2(t-q_{k,j})}{\sqrt{2}\s} }  \sum_{k=1}^{n(q_{k,l})} \eee^{\sqrt{2}  x_k(q_{k,l})} \Big(\sum_{\substack{k'\leq n^{(1)}(t-q_{k,j}),\\j'\leq n^{(2)}(t-q_{k,j})}} \eee^{(\s+i\rho \t)x_{k'}^{(1)}(t-q_{k,j})+(\s-i\rho\t) x_{j'}^{(2)}(t-q_{k,j})}  \Big)^{ \frac{1}{\sqrt{2}\s}}\Big]
.
 \ee
 Next, we bound \eqv(m.6) from above, using Jensen's inequality for the conditional expectation ($\frac{1}{\sqrt{2}\s}<1$), by
\begin{multline}
\Eq(m.6.1) 
\E\Big[ \eee^{-\frac{(1-\rho^2)\t^2(t-q_{k,j})}{\sqrt{2}\s}}  \sum_{k=1}^{n(q_{k,l})} \eee^{\sqrt{2}  x_k(q_{k,l})} 
\\
\times \Big(\E\Big[\sum_{\substack{k'\leq n^{(1)}(t-q_{k,j}),\\j'\leq n^{(2)}(t-q_{k,j})}} \eee^{(\s+i\rho \t)x_{k'}^{(1)}(t-q_{k,j})+(\s-i\rho \t)x_{j'}^{(2)}(t-q_{k,j})} ~\Big\vert~ \mathcal{F}_{q_{k,l}} \Big]\Big)^{ \frac{1}{\sqrt{2}\s}}\Big]
  ,
\end{multline}
  where $\FF_{q_{k,l}}$ is the $\s$-algebra generated by the BBM $X$ up to time $q_{k,l}$, in particular we condition on $q_{k,l}$.
  Calculating the inner expectations in \eqv(m.6.1), gives
  \be\Eq(m.7)
  \begin{aligned}
  &\E\Big[\sum_{k'\leq n^{(1)}(t-q_{k,j}),j'\leq n^{(2)}(t-q_{k,j})} \eee^{(\s+i\rho \t)x_{k'}^{(1)}(t-q_{k,j})+(\s-i\rho \t)x_{j'}^{(2)}(t-q_{k,j})} ~\Big\vert~ \FF_{q_{k,l}} \Big]\\
  &=K\eee^{2(t-q_{k,j})} \int_{-\infty}^\infty \int_{-\infty}^\infty \dd y ~ \dd y' ~ \eee^{(\s+i\t)y+(\s-i\t)y'} \eee^{-\frac{y^2+y'^2}{2(t-q_{k,j})}}\frac{1}{2\pi (t-q_{k,j})} \\
  &=K \eee^{(\s^2-\rho^2\t^2)(t-q_{k,j})+2(t-q_{k,j})}
  \end{aligned}
  \ee
  by completing the square.
  Hence, \eqv(m.6.1) is equal to
  \begin{align}
  \Eq(m.8) 
     K \E\Big[ \eee^{\frac{(\s^2- \t^2+2)(t-q_{k,j})}{\sqrt{2}\s} }  \sum_{k=1}^{n(q_{k,l})} \eee^{\sqrt{2}  x_k(q_{k,l})}\Big]
      &=
      K\int_0^t \dd q ~ \eee^{q}\eee^{\frac{(\s^2- \t^2+2)(t-q )}{\sqrt{2}\s} }  \int_{-\infty}^{\infty} \dd x ~ \eee^{\sqrt{2}x-\frac{x^2}{2q}}\frac{1}{\sqrt{2\pi q}}\nonumber\\
      &= K\int_0^t  \dd q ~ \eee^{\frac{(\s^2- \t^2+2)(t-q )}{\sqrt{2}\s} } \eee^{2q},
  \end{align}
  by computing the Gaussian integral. Using \eqv(m.8) and noticing that the normalization factor in \eqv(mart.1) is equal to $\eee^{\frac{-2t-(\s^2-\t^2)t}{\sqrt{2}\s}}$, we bound the $\frac{\sqrt 2}{\s}$-moment of $\mathcal{M}_{\s,\t}(t)$ by
\begin{align}
\Eq(m.9) 
K\int_0^t  \dd q ~\eee^{\frac{(\s^2- \t^2+2)(t-q )-2t-(\s^2-\t^2)t}{\sqrt{2}\s} } \eee^{2q}
=K \int_0^t \dd q ~ \eee^{\frac{(\t^2-(\s-\sqrt{2})^2)q}{\sqrt{2}\s}} 
.
\end{align}
For $|\t|+|\s|<\sqrt{2}$, the right-hand side\ of \eqv(m.9) is uniformly bounded
by some constant $C$. Since $\MM_{\s,\t}(t)$ is bounded in $L^p$ for some $p>1$,
the a.s.\ limit exists and the convergence also holds in $L^1$. Moreover, 
$\E[\MM_{\s,\t}(t)]=1$ and hence the limit is non-degenerate.
\end{proof}

Next, we turn to proving the central limit theorem when $\s<1/\sqrt{2}$.

\begin{proof}[Proof of Theorem~\thv(CLT.1)] 
  We start with the proof of \eqv(clt.1).
Let
\begin{equation}\Eq(clt.3)
a_k(r) := \eee^{-r\left(1+\frac{\s^2}{2}-\t^2\right)}e^{\s x_k(r)+i\t y_k(r)}.
\end{equation}
Then, we can rewrite $\MM_{\s,\t}(t)$ as
\begin{equation}\Eq(clt.4)
\MM_{\s,\t}(t+r)=\sum_{k=1}^{n(r)} a_k(r)\MM_{\s,\t}^{(k)}(t ),
\end{equation}
where $\MM_{\s,\t}^{(k)}(t)$ are i.i.d.\ copies of $\MM_{\s,\t} (t )$. Hence,
conditional on $\mathcal{F}_r$, $\MM_{\s,\t}(t)$ can be written as a sum of
independent random variables. To prove a CLT, we want to use the two-dimensional
Lindeberg-Feller condition (conditional on $\mathcal{F}_r$). First, we take the
limit $t\uparrow \infty$. For $\s<1/\sqrt{2}$ and $\beta\in B_1$, we have $\s^2+\t^2<1$. Then, by
\cite[Proposition~A.1]{HK15}, $\MM_{\s,\t}^{(k)}(t)$ is $L^2$-bounded and
\begin{equation}\Eq(clt.7)
\lim_{t\uparrow \infty}\E\left[\left \vert \MM_{\s,\t}^{(k)}(t) \right\vert^2  \right]=C_1.
\end{equation}
Hence, the a.s.\  limit $ \MM_{\s,\t}$ exists in $L^2$ and as $t\uparrow \infty$ the right-hand side of \eqv(clt.4) converges a.s.\ to
\begin{equation}\Eq(clt.4')
\MM_{\s,\t}=\sum_{k=1}^{n(r)} a_k(r)\MM_{\s,\t}^{(k)},
\end{equation}
where $\MM_{\s,\t}^{(k)}$ are i.i.d.\ copies of $\MM_{\s,\t}$. To compute the variance of \eqv(clt.4'), consider
\begin{equation}\Eq(clt.5)
\sum_{k=1}^{n(r)} \E\left[\left \vert a_k(r)\MM_{\s,\t}^{(k)} \right\vert^2 ~\Big\vert~ \mathcal{F}_r\right].
\end{equation}
\eqv(clt.5) is equal to
\begin{equation}\Eq(clt.6)
\sum_{k=1}^{n(r)} \vert a_k(r) \vert^2 \E\left[\left \vert \MM_{\s,\t}^{(k)} \right\vert^2  \right]=C_1\sum_{k=1}^{n(r)} \vert a_k(r) \vert^2,
\end{equation}
by \eqv(clt.7). Now,
\begin{equation}\Eq(clt.8)
C_1\sum_{k=1}^{n(r)} \vert a_k(r) \vert^2=C_1\sum_{k=1}^{n(r)} \eee^{2\s x_k(r)-2r\left(1+\frac{\s^2}{2}-\frac{\t^2}{2}\right)}= C_1 \MM_{2\s,0}(r) \eee^{-r\left(1-(\s^2+\t^2)\right)}
.
\end{equation}
\eqv(clt.8) together with  the extra rescaling in \eqv(clt.1),
\begin{equation}\Eq(clt.9)
C_1 \eee^{(1-\s^2-\t^2)r}\sum_{k=1}^{n(r)} \vert a_k(r) \vert^2= C_1 \MM_{2\s,0}(r),
\end{equation}
which converges a.s. as $r\uparrow \infty$ to $C_1 \MM_{2\s,0}$

It remains to check the Lindeberg-Feller condition. We set
 
\begin{equation} \Eq(def.b)
b_k(r) := a_k(r)\eee^{-(1-\s^2-\t^2)r}.
\end{equation}
 Let $\e >0$ and consider
\begin{align}\Eq(clt.1.21)
\frac{1}{ C_1 \MM_{2\s,0}(r)}\sum_{i=1}^{n(r)} \E\Big[ & \left\vert b_k(r)\left(\MM_{\s,\t}^{(k)}-1\right) \right\vert^2
\nonumber
\\
& \times \1\{\vert b_k(r)\left(\MM_{\s,\t}^{(k)}-1\right)\vert > \e\sqrt{C_1\MM_{2\s,0}(r)}\} ~\Big\vert~ \mathcal{F}_r\Big]
.
\end{align}
We rewrite \eqv(clt.1.21) as
\begin{multline}\Eq(clt.1.24)
\frac{1}{ C_1 \MM_{2\s,0}(r)}\sum_{i=1}^{n(r)} b_k(r) \bar b_{k}(r)\E\Big[\left\vert \left(\MM_{\s,\t}^{(k)}-1\right) \right\vert^2
\\
\times \1\{\vert \left(\MM_{\s,\t}^{(k)}-1\right)\vert^2 > \e^2 \vert b_k(r)\vert^2 C_1 \MM_{2\s,0}(r)\} \Bigm| \mathcal{F}_r\Big]
.
\end{multline}
We consider for a fixed $k$
\begin{equation}\Eq(clt.1.25)
\E\left[\left\vert\left( M_{\s,\t}^{(k)}-1\right)\right\vert^2\1\{\vert \left(M_{\s,\t}^{(k)}-1\right)\vert^2>\e^2 \vert b_k(r)\vert^{-2} C^2 M_{2\s,0}(r)\} \Bigm\vert \mathcal{F}_r\right]
.
\end{equation}
Using again that by \cite[Proposition~A.1]{HK15}
\begin{equation}\Eq(clt.1.22)
 \E\left[ \left\vert \left(\MM_{\s,\t}^{(k)}-1\right)\right\vert^2\right] = C_1<\infty
,
\end{equation}
we have that \eqv(clt.1.25) converges to zero as $r \uparrow \infty$ if
\begin{equation}\Eq(clt.1.26)
\vert b_k(r)\vert^{-2} C_1 \MM_{2\s,0}(r) \underset{r\uparrow \infty}{\longrightarrow} \infty
.
\end{equation}
Observe that $\MM_{2\s,0}(r)$ is a $L^2$-bounded
martingale with mean one, if $\s<1/\sqrt{2}$. Hence, it converges a.s.\ and in
$L^1$. Consider
\begin{equation}\Eq(clt.1.27)
\vert b_k(r)\vert^{-2}=\eee^{-2\s x_k(r)+2\left(\frac{1}{2}+\s^2\right)r},
\end{equation}
since $x_k(r)<\sqrt{2}r$ a.s.\ (by Lalley-Selke argument in \cite{LS}). On this event, we have
\begin{equation}\Eq(clt.1.28)
\vert b_k(r)\vert^{-2}\geq \eee^{(-2\sqrt 2 \s +1+2\s^2)r}=\eee^{(1-\sqrt{2}\s)^2r},
\end{equation}
which converges to infinity as $r\uparrow \infty$. Hence,  \eqv(clt.1.26) holds a.s.

 \end{proof}

\section{Proof of CLT for phase $B_3$}
\label{sec:b3}

In this section, we deal with phase $B_3$ and prove Theorem~\ref{clt.B3}.

\subsection{Second moment computations} 
 We start by controlling the second moment of $N_{\s,\t}(t)$ defined in
\eqv(clt.2.1)  in phase $B_3$ and its appropriately scaled version 
\begin{align}
\hat N_{\s,\t}(t):=t^{-1/2}N_{\s,\t}(t)
\end{align}
on the boundary $B_{1,3}$.
\begin{lemma}\TH(Thm.secondmoment)
It holds:
\begin{itemize}

\item[(i)] For $\b\in B_3$ or $\b \in B_{2,3}\setminus \{(1+i)/\sqrt{2}\}$ and any $\rho\in [-1,1]$, 
\begin{equation}\Eq(mom.1)
\lim_{t\uparrow \infty}\E\left[\vert N_{\s,\t}(t)\vert^2\right]=C_2 < \infty,
\end{equation}
for some positive constant $0<C_2<\infty$\footnote{$C_2$ depends on $\s$ and $\t$ but not on $\rho$. We do not make this dependence explicit in our notation.} .

\item[(ii)] For $\b\in B_{1,3}$ or $\b=\frac{1}{\sqrt{2}}(1+i)$ and any $\rho\in [-1,1]$,
\begin{equation}\Eq(mom.1.1)
\lim_{t\uparrow \infty}\E\left[\vert \hat N_{\s,\t}(t)\vert^2\right]=C_3 < \infty,
\end{equation}
for some positive constant $0<C_3<\infty$.
\end{itemize}
\end{lemma}
\begin{proof}
\paragraph{(i)}
We have
\begin{align}
\Eq(mom.2)
\E\Big[\vert N_{\s,\t}(t) \vert^2\Big]
=\eee^{-2t\left(1/2+\s^2\right)}\E\Big[ 
\sum_{k,l=1}^{n(t)} \eee^{\s \left(x_k(t)+x_l(t)\right)+i\t \left(y_k(t)-y_l(t)\right)}
\Big].
\end{align}
Using Representation \eqv(cor.1), we rewrite the right-hand side of \eqv(mom.2) as
\begin{align}
\Eq(mom.3)
\eee^{-2t\left(1/2+\s^2\right)}\E\Big[ 
\sum_{k,l=1}^{n(t)} \eee^{\bar \l x_l(t)+\l x_k(t)+i\t \sqrt{1-\rho^2} (z_k(t)-z_l(t))}
\Big],
\end{align}
where $\l= \s+i\rho \t$ and $(z_k(t))_{k\leq n(t)}$ are the particles of a BBM
on $\mathbb{T}_t$ that is independent from $X(t)$. By conditioning on
$\mathcal{F}^{\mathbb{T}_t}$, we have that \eqv(mom.3) is equal
to
\begin{align}
\Eq(mom.4)
\eee^{-2t\left(1/2+\s^2\right)}\E\Big[ \eee^{-(1-\rho^2)\t^2\left(t-d(x_k(t),x_l(t))\right)}
\sum_{k,l=1}^{n(t)} \eee^{\bar \l x_l(t)+\l x_k(t)}
\Big].
\end{align}
The expectation in \eqv(mom.4) is equal to
\begin{align}\Eq(mom.5)
&K \int_0^t \ddd q ~ \eee^{2t-q-(1-\rho^2)\t^2(t-q)}\int_{-\infty}^{\infty} \frac{\ddd x}{\sqrt{2\pi q}} \int_{-\infty}^{\infty}\frac{\ddd y}{\sqrt{2\pi (t-q)}}\nonumber\\
&\quad\times \int_{-\infty}^{\infty}\frac{\ddd y'}{\sqrt{2\pi (t-q)}} \eee^{2\s x +\s(y+y')+i\t\rho (y-y')}\eee^{-\frac{y^2+y'^2}{2(t-q)}}\eee^{-x^2/2q}.
\end{align}
Computing first the integrals with respect to $y$ and $y'$, we get that \eqv(mom.5) is equal to
\begin{align}\Eq(mom.6)
&K \int_0^t \ddd q ~\eee^{2t-q-(1-\rho^2)\t^2(t-q)+(\s^2-\rho^2\t^2)(t-q)}\int_{-\infty}^{\infty} \frac{\ddd x}{\sqrt{2\pi q}}  \eee^{2\s x} \eee^{-x^2/2q}\nonumber\\
&= K \int_0^t \ddd q ~ \eee^{2t-q-\t^2(t-q)+\s^2(t-q)} \eee^{2\s^2q}.
\end{align}
Plugging \eqv(mom.6) back into \eqv(mom.4), we get that \eqv(mom.4) is equal to
\begin{align}\Eq(mom.7)
&\eee^{-2t\left(1/2+\s^2\right)}K \int_0^t \ddd q ~\eee^{2t-q-\t^2(t-q)+\s^2(t-q)} \eee^{2\s^2q}\nonumber\\
&=K \int_0^t \ddd q ~\eee^{(t-q)(1-\t^2-\s^2)} = K\int_0^t \ddd q' ~\eee^{q'(1-\t^2-\s^2)} \nonumber\\
&=\frac{K}{1-\t^2-\s^2} \left(\eee^{t(1-\t^2-\s^2)}-1\right)
.
\end{align}
As $t\uparrow \infty$, the term in \eqv(mom.7) converges to $\frac{K}{\t^2+\s^2-1}$, which we call $C_2$ from now on.

\paragraph{(ii)}  Proceeding as in Part (i), we get that 
\begin{align}
\Eq(mom.4.1)
\E \left[\vert \hat N_{\s,\t}(t)\vert^2\right]=t^{-1}\eee^{-2t\left(1/2+\s^2\right)}\E\Big[ \eee^{-(1-\rho^2)\t^2\left(t-d(x_k(t),x_l(t))\right)}
\sum_{k,l=1}^{n(t)} \eee^{\bar \l x_l(t)+\l x_k(t)}
\Big].
\end{align}
Plugging  \eqv(mom.6)  into \eqv(mom.4.1), we get that \eqv(mom.4.1) is equal to
\be\Eq(mom.7ak)
Kt^{-1}\int_0^t \ddd q ~\eee^{(t-q)(1-\t^2-\s^2)} =Kt^{-1}\int_0^t \ddd q = K,
\ee
since $\s^2+\t^2=1$ in $B_{1,3}$.

\end{proof}

\subsection{Constrained moment computation in $B_3$}

In this section, we continue our preparations for the proof of Theorem~\ref{clt.B3}. These consist of computing constrained moments. The following two Lemmata ensure that we can introduce the desired constraint.

\begin{lemma}\TH(Lem.const1)
 Let $\b\in B_3$. Then for all $\e>0$ and $\delta>0$, uniformly for all $t$ large enough, there exists $A_0$ such that for all $A>A_0$   
\be\Eq(con.1)
\P\Big\{\Big\vert\sum_{k=1}^{n(t)}e^{\s x_k(t)+i\t y_k(t)-\left(\frac{1}{2}+\s^2\right)t}\1_{\{x_k(t)>2\s t+A\sqrt{t}\}}\Big\vert>\delta\Big\}<\e
.
\ee
 \end{lemma}
 \begin{proof}
Using a second moment Chebyshev inequality, we bound the probability in \eqv(con.1) from above by
 \be\Eq(con.2)
 \eee^{-2t\left(1/2+\s^2\right)}\E\Big[ 
\sum_{k,l=1}^{n(t)} \eee^{\s (x_k(t)+x_l(t))+i\t (y_k(t)-y_l(t))}\1\{x_k(t),x_l(t)>2\s t+A\sqrt{t}\}
\Big].
 \ee
Continuing as in the proof of Lemma~\thv(Thm.secondmoment), we rewrite \eqv(con.2) as
\be\Eq(con.3)
\eee^{-2t\left(1/2+\s^2\right)}\E\Big[ \eee^{-(1-\rho^2)\t^2\left(t-d(x_k(t),x_l(t))\right)}
\sum_{k,l=1}^{n(t)} \eee^{\bar \l x_l(t)+\l x_k(t)}\1\{x_k(t),x_l(t)>2\s t+A\sqrt{t}\}
\Big].
 \ee
We rewrite the expectation in \eqv(con.3)
\begin{multline}
\Eq(con.4)
K \int_0^t \ddd q ~ \eee^{2t-q-(1-\rho^2)\t^2(t-q)}\int_{-\infty}^{\infty} \frac{\ddd x}{\sqrt{2\pi q}} \int^{\infty}_{2\s t+A\sqrt{t}-x}\frac{\ddd y}{\sqrt{2\pi (t-q)}}
\\
\quad\times \int_{2\s t-x}^{\infty}\frac{\ddd y'}{\sqrt{2\pi (t-q)}} \eee^{2\s x+A\sqrt{t} +\s(y+y')+i\t\rho (y-y')}\eee^{-\frac{y^2+y'^2}{2(t-q)}}\eee^{-x^2/2q}.
\end{multline}
Observe that by the computations in Lemma~\thv(Thm.secondmoment) for $r$ sufficiently large
\begin{multline}
\Eq(con.5)
K \int_0^{t-r}\ddd q ~ \eee^{2t-q-(1-\rho^2)\t^2(t-q)}\int_{-\infty}^{\infty} \frac{\ddd x}{\sqrt{2\pi q}} \int^{\infty}_{2\s t+A\sqrt{t}-x}\frac{\ddd y}{\sqrt{2\pi (t-q)}}
\\
\quad\times \int_{2\s t+A\sqrt{t}-x}^{\infty}\frac{\ddd y'}{\sqrt{2\pi (t-q)}} \eee^{2\s x +\s(y+y')+i\t\rho (y-y')}\eee^{-\frac{y^2+y'^2}{2(t-q)}}\eee^{-x^2/2q}< \e/2.
\end{multline}
Hence, it suffices to take consider the integration domain $q>t-r$. Now, $\P(y>r)<\eee^{-r/2}$.
To have $x+y>2\s t+A\sqrt t$ on that event, $x>2\s t + A\sqrt t-r$ must hold. By
inserting this into the second moment, we have that the bound is smaller than
$\e/2$ for $A$ sufficiently large.
\end{proof}

\begin{lemma}\TH(Lem.const2)
Let $\b\in B_3$, $\rho\in[-1,1]$ and $\g>\frac{1}{2}$. Let $A>0$. 
Then, for all $\e>0$ and $d>0$, there exists $r_0>0$ such that, for all $r>r_0$, uniformly for all $t$ sufficiently large,
\begin{multline}\Eq(con.10)
\P\Big\{\Big\vert\sum_{k=1}^{n(t)}\eee^{\s x_k(t)+i\t y_k(t)-\left(\frac{1}{2}+\s^2\right)t}
\\
\times \1\{x_k(t)<2\s t+A\sqrt{t}, \exists s \in [r,t]\colon x_k(s)>2\s s+s^{\gamma}\}\Big\vert>\delta\Big\}
<\e
.
\end{multline}
\end{lemma}
\begin{proof}
We use again a second moment bound. Similarly to the proof of Lemma~\thv(Lem.const1), we bound the probability in \eqv(con.10) from above by
\begin{multline}
\Eq(ny.40)
\eee^{-2t\left(1/2+\s^2\right)}\E\Big[ \eee^{-(1-\rho^2)\t^2\left(t-d(x_k(t),x_l(t))\right)}
\sum_{k,l=1}^{n(t)} \eee^{\bar \l x_l(t)+\l x_k(t)} 
\\
\times \1\{x_k(t),x_l(t)<2\s t+A\sqrt{t}, \exists s \in [r,t]\colon x_k(s)>2\s s+s^{\gamma}, 
\\
\exists s' \in [r,t]\colon x_l(s')>2\s s'+(s')^{\gamma}\}\Big].
\end{multline}
By only keeping track of the path event for one of the particles, we get that \eqv(ny.40) is bounded from above by
\begin{multline}
\Eq(ny.41)
\eee^{-2t\left(1/2+\s^2\right)}\E\Big[ \eee^{-(1-\rho^2)\t^2\left(t-d(x_k(t),x_l(t))\right)}
\sum_{k,l=1}^{n(t)} \eee^{\bar \l x_l(t)+\l x_k(t)}
\\
\times
\1\{x_k(t),x_l(t)<2\s t+A\sqrt{t}\;\exists s \in [r,t]\colon x_k(s)>2\s s+s^{\gamma}\}
\Big].
\end{multline}
We rewrite \eqv(ny.41) as
\be\Eq(ny.42)
\begin{aligned}
 & K \int_0^t \dd q ~ \eee^{2t-q} \eee^{2t-q-(1-\rho^2)\t^2(t-q)} \E\left[ \eee^{\bar \l x_1(t)+\l(x_1(q)+ x_2(t-q))} \right.
\\
 & \left.  \times \1 \{x_1(t),x_1(q)+x_2(t-q)<2\s t+A\sqrt{t}, \exists s \in [r,t]\colon x_1(s)>2\s s+s^{\gamma}\} \right]
 ,
 \end{aligned}
\ee
where $x_1(\cdot)$ is a standard Brownian motion and $x_2(t-q)$ is an independent $\mathcal{N}(0,t-q)$ distributed random variable.
Calculation of the expectation in \eqv(ny.42) with respect to $x_2(t-q)$ yields 
\begin{multline}
\Eq(ny.43)
K \int_0^t \ddd q~ \eee^{2t-q}  \eee^{-\frac{\lambda^2}{2(t-q)}}\eee^{2t-q-(1-\rho^2)\t^2(t-q)} \E\left[\eee^{\bar \l x_1(t)+\l x_1(q)}
\right.
\\
\times \left. \1\{x_1(t)<2\s t+A\sqrt{t}, \exists s \in [r,t]\colon x_1(s)>2\s s+s^{\gamma}\} \right] 
.
\end{multline}
As in the proof of Lemma~\thv(Lem.const1), we can first choose $r_1$ large enough such that the above integral from $0$ to $t-r_1$ is bounded by $\e/3$. Moreover, $x_1(t)=x_1(q)+ \tilde{x}(t-q)$, where $\tilde{x}(t-q)$ is normal distributed with mean zero and variance $t-q$ that is independent from $x_1(s)$ for $s \leq q$. Then, for all $R>R_2$, 
\be\Eq(sf.1)
\P\{\vert\tilde{x}(t-q)\vert > R \}<\frac{\e}{3}.
\ee

Observe that the intersection of the event $\{\tilde{x}(t-q)>R\}$  and the event in the indicator in \eqv(ny.43)  is contained in the event 
\begin{multline}
\{x_1(t)<2\s t+A\sqrt{t}, \exists s \in [r,t]\colon x_1(s)>2\s s+s^{\gamma},\tilde{x}(t-q)>R\}
\\
\subset \Big\{x_1(s)-\frac{s}{q}x_1(q) <s^{\gamma}-\frac{(A\sqrt{t}-R) s}{q} \Big\}.
\end{multline} 
Using that $x_1(s)-\frac{s}{q}x_1(q)=\xi_k(s)$ is a Brownian bridge that is independent from $x_1(q)$ and also from $\tilde{x}(t-q)$, we bound \eqv(ny.43) from above by
\begin{multline}
\Eq(ny.44)
K \int_0^t \dd q ~ \eee^{2t-q} \eee^{-\frac{\lambda^2}{2(t-q)}}\eee^{2t-q-(1-\rho^2)\t^2(t-q)} \E\left[\eee^{\bar \l x_1(t)+\l x_1(q)}  \right]
\\
\quad\times \P\Big\{\exists s \in[r,t-r]\colon \xi(s)>s^{\gamma}-\frac{(A\sqrt{t}-R) s}{q}\Big\}
.
\end{multline}
By the same computations as in \eqv(mom.5) and \eqv(mom.6), we can bound \eqv(ny.44) from above by
\be \Eq(ny.45)
C_2 \P\Big\{\exists s \in[r,t-R-r]\colon \xi(s)>s^{\gamma}-\frac{(A\sqrt{t}-R) s}{t-R}\Big\}
.
\ee
 It is a well known fact for Brownian bridges (see, e.g., \cite[Lemma~2.3]{BovHar13} for a precise statement) that by choosing $r$ sufficiently large
\eqv(ny.45) can be made as small as we want. This finishes the proof of Lemma~\thv(Lem.const2).
\end{proof}
Define
\begin{multline}\Eq(con.20)
N_{\s,\t}^{c,A}(t) := \sum_{k=1}^{n(t)} \eee^{-t(1/2+\s^2 )} \eee^{\s x_k(t)+i\t y_k(t)}
\\
\times\1\{x_k(t)<2\s t+A\sqrt{t}, \forall s \in [r,t]\colon x_k(s)\leq2\s s+s^{\gamma}\}
.
\end{multline}
The following lemma provides the asymptotics for all moments of \eqref{con.20} in the $t \to \infty$ limit.
\begin{lemma}[Moment asymptotics]\TH(Lem.const3) 
Consider a branching Brownian motion with binary splitting. 
For $\b\in B_3$, for any $A>0$
\be\Eq(con.100)
\lim_{t\to\infty}\E\left[\left\vert N_{\s,\t}^{c,A}(t)\right\vert^2\right]=C_{2,A}
,
\ee
with $\lim_{A\to\infty} C_{2,A}=C_2$ and, for $k\in \N$, we have
\be \Eq(con.101)
\lim_{r\uparrow \infty}\lim_{t\to \infty} \E\left[\left\vert N_{\s,\t}^{c,A}(t)\right\vert^{2k} ~\vert~ \mathcal{F}_r\right] =k! (C_{2,A}\M_{2\s,0})^k \quad \mbox{a.s. and in }L^1.
\ee
Moreover, for $k'<k$,
\be\Eq(con.102)
\lim_{r\uparrow \infty}\lim_{t\to \infty} \E\left[N_{\s,\t}^{c,A}(t)^k\overline{ N_{\s,\t}^{c,A}(t)}^{k'} ~\big\vert~ \mathcal{F}_r \right] =0\quad \mbox{a.s. and in }L^1.
\ee
 
\end{lemma}
\begin{proof}
We proceed by induction over $k \in \N$.
For $k=1$, we observe that the claim follows directly from Lemma~\thv(Thm.secondmoment) together with the second moment computation done in the proof of Lemma~\thv(Lem.const1) and Lemma~\thv(Lem.const2).

To bound the $2k$-moment, we rewrite \eqv(con.101) as
\be\Eq(ny.1) 
\begin{aligned}
\E\Big[\sum_{l_1,\dots, l_{2k}\leq n(t)}\prod_{j=1}^{2k} & \eee^{-t(1/2+\s^2 )} \eee^{\s x_{l_j}(t)+i\t y_{l_j}(t)}
\\
&
\times \1\{x_{l_j}(t)<2\s t+A\sqrt{t}, \forall s \in [r,t] \colon x_{l_j}(s)\leq2\s s+s^{\gamma}\}\Big] 
.
\end{aligned}
\ee
  For $l_1,\dots, l_{2k}\leq n(t)$, we can find a matching using the following algorithm:
  \begin{itemize}
  \item[1.] Choose the two labels $j,j'$ such that $d(x_{l_j},x_{l_{j'}})$ is maximal. Call them $l_1$ and $l_{\sigma(1)}$ from know on.
  \item[2.] Delete them.
  \item[3.]  Pick $l_j$ in the remaining set and match it with the remaining $l_{j'}$ such that $d(x_{l_j},x_{l_{j'}})$ is maximal. Iterate.
    \end{itemize}
    The pairs obtained in this way we call $(l_1,l_{\s(1)}), \dots, (l_k, l_{\s(k)})$. 
    We rewrite \eqv(ny.1) as
    \be\Eq(ny.2)
    \begin{aligned}
    &\E\Big[\sum_{l_2,\dots, l_{k}\leq n(t)}\prod_{j=2}^{k} \eee^{-t(1+2\s^2 )} \eee^{\s \left(x_{l_j}(t)+x_{l_{\s(j)}}(t)\right)+i\t \left(y_{l_j}(t)+y_{l_{\s(j)}}(t)\right)}
    \\
    & \quad \times
    \1\{ x_{l_{\s(j)}}(t),x_{l_j}(t)<2\s t+A\sqrt{t}, \forall s \in [r,t] \colon x_{l_{\s(j)}}(s), x_{l_j}(s) \leq2\s s+s^{\gamma}\} 
    \\
    &  \quad \times \eee^{-t(1+2\s^2 )} \eee^{\s (x_{l_1}(t)+x_{l_{\s(1)}}(t))+i\t (y_{l_1}(t)-y_{l_{\s(1)}}(t))}
    \\
   & \quad \times \1\{x_{l_{\s(1)}}(t),x_{l_1}(t)<2\s t+A\sqrt{t}, \forall s \in [r,t] \colon  x_{l_{\s(1)}}(s),x_{l_1}(s)\leq2\s s+s^{\gamma}\}
 \Big]
 .
    \end{aligned}
    \ee
    Using \eqv(cor.1), we can rewrite for $j\in\{1,\s(1)\}$    
    \be \Eq(ny.3)
    y_{l_j}(t)= \rho y_{l_j}(t)+ \sqrt{1-\rho^2} z_{l_j}(t),
    \ee
     where $(z_k(t))_{k\leq n(t)}$ are particles of a BBM on the same Galton-Watson tree as $(x_k(t))_{k\leq n(t)}$ but independent from it. 
     Observe that using  the requirement that $d(x_{l_1},x_{l_{\s_1}})$ is chosen maximal, we have  
     \be \Eq(ny.4)
     \begin{aligned}
     i\t (y_{l_1}(t)-y_{l_{\s(1)}}(t))&=i\sqrt{1-\rho^2} \tau \left(z_{1}(t-d(x_{l_1}(t),x_{l_{\s_1}}(t)))-z_{2}(t-d(x_{l_1}(t),x_{l_{\s_1}}(t))\right)
     \\
   & \quad+i \tau \rho \left(x_{l_1}(t)-x_{l_{\s(1)}}(t)\right),
   \end{aligned}
     \ee
     where $z_1, z_2$ are two independent $\mathcal{N}(0,(t-d(x_{l_1}(t),x_{l_{\s_1}}(t)))) $-distributed random variables.
     Plugging \eqv(ny.4) into \eqv(ny.2) and computing the expectation with respect to $z_1,z_2$, we obtain
     \be\Eq(ny.5)
     \begin{aligned}
    &\E\Big[\sum_{l_2,\dots, l_{k}\leq n(t)}\prod_{j=2}^{k} \eee^{-t(1+2\s^2)} \exp\left(\s (x_{l_j}(t)+x_{l_{\s(j)}}(t))+i\t (y_{l_j}(t)+y_{l_{\s(j)}}(t))\right)
    \\
    & \quad\quad \times \1\{ x_{l_{\s(j)}}(t),x_{l_j}(t)<2\s t+A\sqrt{t}, \forall s \in [r,t] \colon x_{l_{\s(j)}}(s), x_{l_j}(s)\leq2\s s+s^{\gamma}\}
     \\
     &
     \quad\quad \times \eee^{-t(1+2\s^2 )-\t^2(1-\rho^2)\left(t-d(x_{l_1},x_{l_{\s(1)}})\right)} \eee^{(\s +i\tau\rho)x_{l_1}(t)+(\s -i\tau\rho)x_{l_{\s(1)}}}
     \\
     &
     \quad\quad\times \1\{x_{l_{\s(1)}}(t),x_{l_1}(t)<2\s t+A\sqrt{t}, \forall s \in [r,t] \colon x_{l_{\s(1)}}(s),x_{l_1}(s)\leq2\s s+s^{\gamma}\}
    \Big]
    .
    \end{aligned} 
    \ee
     We decompose 
     \be\Eq(ny.6)
     \begin{aligned}
     x_{l_{\s(1)}}(t)= x_{l_1}d(x_{l_1},x_{l_{\s(1)}}) + x^{(1)}(t-d(x_{l_1},x_{l_{\s(1)}})); 
     \\
      x_{l_1}(t)= x_{l_1}d(x_{l_1},x_{l_{\s(1)}}) + x^{(2)}(t-d(x_{l_1},x_{l_{\s(1)}})),
      \end{aligned}
     \ee
     where $x^{(1)},x^{(2)}$ are two independent $\mathcal{N}(0,t-d(x_{l_1},x_{l_{\s(1)}})) $-distributed random variables. By Step one of our matching procedure, we can plug \eqv(ny.5) into \eqv(ny.6) and compute the expectation with respect to $x^{(1)}$ and $x^{(2)}$, we obtain that \eqv(ny.5) is bounded from above by\footnote{A corresponding lower bound also holds due to the second moment computation in Lemma \thv(Lem.const3).} 
     \be \Eq(ny.7)
     \begin{aligned}
    &\E\Big[\sum_{l_2,\dots, l_{k}\leq n(t)}\prod_{j=2}^{k} \eee^{-t(1+2\s^2 )} \eee^{\s \left(x_{l_j}(t)+x_{l_{\s(j)}}(t)\right)+i\t \left(y_{l_j}(t)+y_{l_{\s(j)}}(t)\right)}
    \\
    & \quad \times \1\{ x_{l_{\s(j)}}(t),x_{l_j}(t)<2\s t+A\sqrt{t}, \forall s \in [r,t] \colon x_{l_{\s(j)}}(s), x_{l_j}(s)\leq2\s s+s^{\gamma}\}
    \\
    &  \quad \times \eee^{-t(1+2\s^2 )-\t^2\left(t-d(x_{l_1},x_{l_{\s(1)}})\right)+\s^2 \left(t-d(x_{l_1},x_{l_{\s(1)}})\right)} \eee^{2\s x_{l_1}d(x_{l_1},x_{l_{\s(1)}})}
    \\
    & \quad \times
    \1\{  \forall s \in [r,d(x_{l_1},x_{l_{\s(1)}})] \colon x_{l_{\s(1)}}(s),x_{l_1}(s)\leq2\s s+s^{\gamma}\}
 \Big]
 .
      \end{aligned} 
    \ee
     We now introduce the event 
     \be\Eq(ny.8)
     \mathcal{A}_r= \left\{\exists s \in [r, d(x_{l_1},x_{l_{\s(1)}})], \exists j\in\{2,\dots,k,\s(2),\dots, \s(k)\} \colon d(x_{l_1}, x_{l_{j}}) =s\right\} 
     .
     \ee
    We can rewrite  \eqv(ny.7) as
    \be \Eq(ny.10)
    \E\left[\ldots \times \1_{\mathcal{A}_r}\right] + \E\left[\ldots \times \1_{\mathcal{A}_r^c}\right] =: J_{\mathcal{A}_r}+    J_{\mathcal{A}_r^c}.
       \ee
    We will prove that the first summand is of a smaller order than the second one.
  We need the following lemma.
  
  \begin{lemma}\TH(Lem.help)
  Let $x,y$ be $\mathcal{N}(0,q)$ distributed random variables. Then, for any $m_1,m_2\geq 1$ and constant $C>0$,
  \be\Eq(ny.11)
  \begin{aligned}
&\E\left[ \eee^{(m_1+2)\sigma x +i\tau m_2 x}  \1\{x<2\s q+ Cq^{\gamma}\}\right]
\\
& \quad 
\underset{q \to \infty}{=} o\left(\eee^{2\s q} \E\left[ \eee^{m_1\sigma x +i\tau m_2 x}  \1\{x<2\s q+ Cq^{\gamma}\}\right]\E\left[\eee^{2\sigma y}  \1\{y<2\s q+ Cq^{\gamma}\}\right]\right)
,
\end{aligned}
 \ee
 and similarly
  \be\Eq(ny.30)
  \begin{aligned}
    &\E\left[ \eee^{(m_1+1)\sigma x +i\tau (m_2+1) x}  \1\{x<2\s q+ Cq^{\gamma}\}\right]
    \\
    & \quad \underset{q \to \infty}{=} o\left(\eee^{2\s q} \E\left[ \eee^{m_1\sigma x +i\tau m_2 x}  \1\{x<2\s q+ Cq^{\gamma}\}\right]\E\left[ \eee^{(\sigma+i\tau) y}  \1\{y<2\s q+ Cq^{\gamma}\}\right]\right)
    .
  \end{aligned}
  \ee
  \end{lemma}
  \begin{proof}
  The l.h.s.\ in \eqv(ny.11) is equal to
  \be\Eq(ny.12)
  \int_{-\infty}^{2\s q+ Cq^{\gamma} } \frac{\dd x}{\sqrt{2\pi q}} ~ \eee^{(m_1+2)\sigma x +i\tau m_2 x} \eee^{-\frac{x^2}{2q}}.
  \ee
  Making a change of variable $y=(m_1+2)\s q + i\tau m_2 q+x$, we obtain that \eqref{ny.12} equals to
  \be\Eq(ny.13)
  \eee^{\left((m_1+2)\s  + i\tau m_2 \right)^2q/2}  \int_{-\infty}^{-m_1\s q- i\tau m_2 q + Cq^{\gamma} } \frac{\dd y}{\sqrt{2\pi q}} ~ \eee^{-y^2/2q}.
  \ee
  For $m_1 \geq 1$, by the Gaussian tail asymptotics, \eqref{ny.13} is bounded from above by
  \be \Eq(ny.14)
  \eee^{2m_1\s^2 q+2\s q+m_2^2\tau^2 q} \eee^{m_1\s Cq^{\gamma} }.
  \ee
  The expectation on the right hand side of \eqv(ny.11) is equal to 
  \be\Eq(ny.15)
  \int_{-\infty}^{2\s q+ Cq^{\gamma} } \dd x~\eee^{m_1\sigma x+i\tau m_2 x} \eee^{-\frac{x^2}{2q}}  
  \int_{-\infty}^{2\s q+ Cq^{\gamma} } \dd y~\eee^{2\sigma y}  \eee^{-\frac{y^2}{2q}} 
  .
  \ee
If $m_1>2$, \eqref{ny.15} is asymptotically equal to 
  \be\Eq(ny.16)
  \frac{1}{\sqrt{2\pi (m_1-2)q}}\eee^{2m_1\s^2 q-2\s^2 q+m_2^2\tau^2 q}\eee^{2\s^2 q}\eee^{(m_1-2)\s Cq^{\gamma} }
  .
  \ee
Comparing \eqv(ny.16) with \eqv(ny.14) yields the claim of Lemma~\eqv(ny.11). For $m_1=1$ or $m_2=1$, we bound the integral in \eqv(ny.15) by  $\eee^{(m_1\s +i\tau m_2)^2q/2+2\s^2q/2}$.

The proof of \eqv(ny.30) follows along the same lines. 
\end{proof}
We continue the proof of Lemma~\thv(Lem.const3). Consider $ J_{\mathcal{A}_r}$. 
Consider the skeleton generated by the leaves $l_1,l_{\s(1)}, \dots, l_k, l_{\s(k)}$ of the Galton-Watson tree. By $\mbox{path}(\cdot)$ we denote the unique path from a leave $\cdot$ to the root. To each edge in the Galton-Watson tree, we associate the following number
\be\Eq(ny.17)
m(e) := \sum_{j\in\{1,\s(1), \dots, k, \s(k)\}} \1_{e  \subset \mbox{path}(l_j)}.
\ee
  
\begin{figure}[tbph]
\centering
\includegraphics[width=0.7\linewidth]{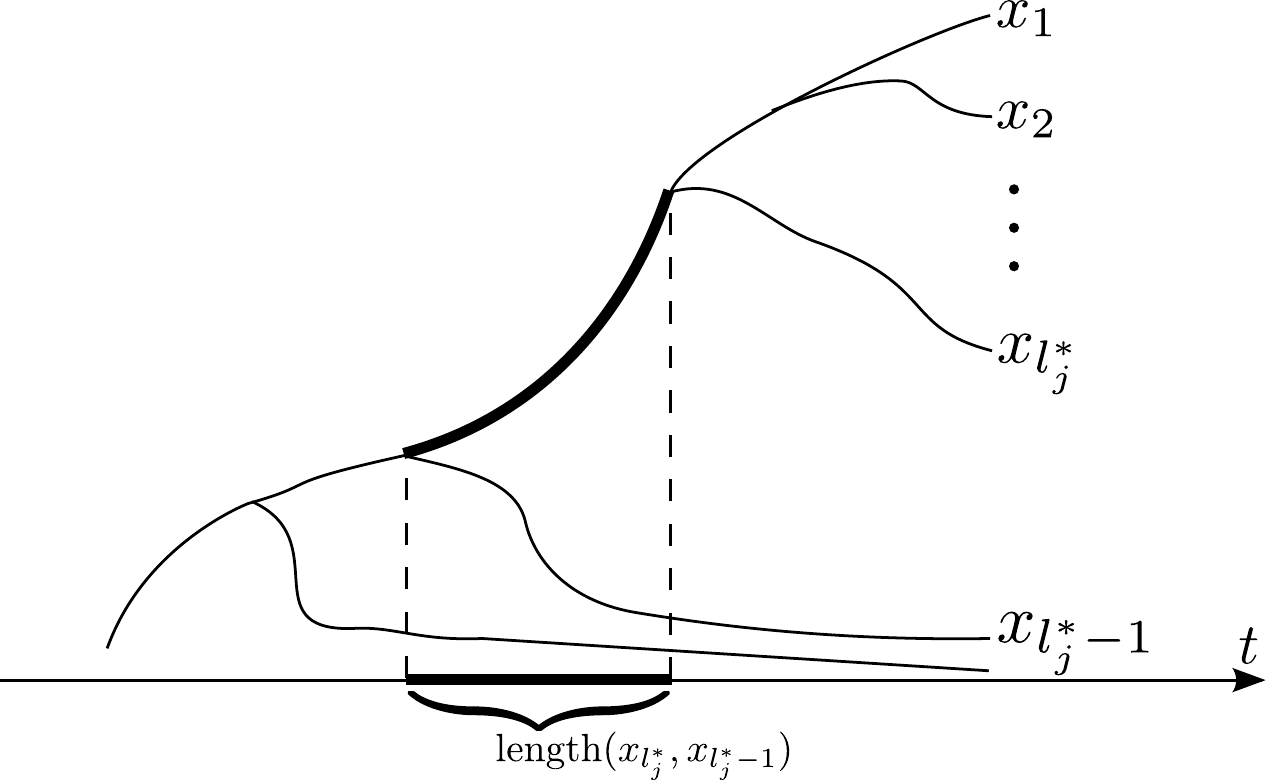}
\caption{Illustration of the notion of $\length(\cdot, \cdot)$ as defined in \ref{eq:length}}
\label{fig:length}
\end{figure}
For $k,j \in [n(t)]$, define (cf.~Fig.~\ref{fig:length})
\begin{align}\label{eq:length}
\length(x_k(t), x_j(t)) & :=  d(x_1(t), x_k(t))-d(x_1(t), x_j(t)), \quad t \in \R_+,
. 
\end{align}
Looking at the path of $x_{l_1}(t)$, the quantity $m(\cdot)$ for $e\subset
\mbox{path}(l_1)$ before time $d(x_{l_1},x_{l_{j^{*}}})$ where $l_{j*}$ satisfies the following conditions
\begin{itemize}
\item[(i)] $m$ is constant between $l_{j*}$ and $l_{j*-1}$ and the piece has length $>2r$.
\item[(ii)] $\sum_{i=1}^{{j*-1}} \length(x_{l_{i-1}},x_{l_{i}})< (\length (x_{l_{j^{*}-1}},x_{l_{j^{*}}}))^\gamma$,
where $\length$ is defined in the Fig.~\ref{fig:length}.
\end{itemize}
Such a $l_{j^{*}}$ exists for all $t>t_0(r)$ because there are
at most $2k-2$ points, where $m$ it is allowed to change. 
We call the value of $m$ on the path where $l_{j*}$ and $l_{j*-1}$ between $m^*$. $m^*$
corresponds to an time interval $[R, R+\ell]$, where
\begin{align}
\ell = \ell(j^*, t) :=  \length(x_{l_{j^{*}-1}}(t),x_{l_{j^{*}}}(t)).
\end{align}
Then, up to time $R$ the minimal
particle is a.s.~$>-\sqrt{2}R$. Hence,
\be \Eq(ny.19)
x_{l_{j*}}(R+\ell)- x_{l_{j*}}(R)<x_{l_{j*}}(R+\ell) +\sqrt{2}R
.
\ee
Since we compute an expectation conditional on  $x_{l_{j*}}(R+\ell) < 2 \s (R+\ell)+ (R+\ell)^\gamma$,
we obtain on this event
\be \Eq(ny.20)
x_{l_{j*}}(R+\ell)- x_{l_{j*}}(R)< 2 \s (R+\ell)+ (R+\ell)^\gamma+\sqrt{2} R.
\ee
Due to our choice of $j^*$, we have $2 \s R+\sqrt{2}R< C' (\ell)^\gamma$ for some positive constant $C'$.
By taking the expectation with respect to $x_{l_{j*}}(R+\ell)- x_{l_{j*}}(R)$ only, we can write extract from $J_{\mathcal{A}}$ the factor
\be \Eq(ny.21)
\E\left[ \eee^{(m^*\s+ i\tau m' )x_{l_{j*}}(R+\ell)- x_{l_{j*}}(R)}\1\{x_{l_{j*}}(R+\ell)- x_{l_{j*}}(R)< 2\s\ell+ (C'+1) (\ell)^\gamma\}\right].
\ee
By Lemma~\thv(Lem.help), \eqref{ny.21} is  
\be \Eq(ny.22)
\begin{aligned}
o\Big(\eee^{2\s \ell} & \E\left[\eee^{((m^*-2)\s+ i\tau m' )x_{l_{j*}}(R+\ell)- x_{l_{j*}}(R)}\1\{x_{l_{j*}}(R+\ell)- x_{l_{j*}}(R)< 2\s\ell+ (C'+1) (\ell)^\gamma\}\right]
\\
& \quad \times \E\left[ \eee^{2\s (x_{l_{j*}}(R+\ell)- x_{l_{j*}}(R))}\1\{x_{l_{j*}}(R+\ell)- x_{l_{j*}}(R)< 2\s\ell+ (C'+1) (\ell)^\gamma\}\right]\Big),
\end{aligned}
\ee
for $l$ large (which by Assumption~(i) on $l$ corresponds to $r$ large). Note that the quantity, inside the brackets in \eqv(ny.22), corresponds to the same expectation but where in the underlying tree $l_1, l_{\s_1}$ branched off before time $R$.

Iteratively, that leads to 
\be\Eq(ny.23)
J_{\mathcal{A}_r} \underset{t,r \to \infty}{=} o( J_{\mathcal{A}_r^c})
.
\ee
Since $k$ was chosen arbitrary, we know that the main contribution to the $2k$-th moment comes from the term where $l_1,\dots, l_k $ have split before time $r$ for $r$ large enough.
We condition on $\mathcal{F}_r$ and compute:
\be\Eq(ny.24)
\begin{aligned}
&\E\Big[\sum_{l_1,l_2,\dots, l_{k}\leq n(t)}\prod_{j=2}^{k} \eee^{-t(1+2\s^2 )} \eee^{\s (x_{l_j}(t)+x_{l_{\s(j)}})+i\t (y_{l_j}(t)+y_{l_{\s(j)}}(t))}
\\
& \quad \times \1\{ x_{l_{\s(j)}}(t),x_{l_j}(t)<2\s t+A\sqrt{t}, \forall s \in [r,t] \colon x_{l_{\s(j)}}, x_{l_j}\leq2\s s+s^{\gamma}\} 
\\
& \quad \times \1\{\sup_{j,j'\leq k}d(l_j,l_{j'})<r\} ~\Big|~ \mathcal{F}_r\Big]
\\
& =
\E\Big[\sum_{l_1,l_2,\dots, l_{k}\leq n(t)}\prod_{j=2}^{k}   b_{l_j}(r)\overline{b}_{l_{\s(j)} }(r) \E\left[\left(\left((N_{\s,\t}^{\gamma,A}(t-r)\right)^{(j)}\right)^{2}\right] ~\Big|~ \mathcal{F}_r\Big],
\end{aligned}
\ee
where $b_{l_j}(r)$ is defined in \eqv(def.b) and $\left(N_{\s,\t}^{\gamma,A}(t-r)\right)^{(j)}$ are i.i.d. copies of $\left(N_{\s,\t}^{\gamma,A}(t-r)\right)^{(j)}$.
By our second moment computations (Case $k=1$), as mentioned at the beginning of this proof,
\be \Eq(ny.25)
\lim_{t\to \infty} \E\left[\left(\left((N_{\s,\t}^{\gamma,A}(t-r)\right)^{(j)}\right)^{2}\right] = C_{2,A}.
\ee
Moreover, by invariance under permutation (in the labelling procedure),
\be\Eq(ny.26)
\sum_{l_1,l_2,\dots, l_{k}\leq n(t)}\prod_{j=2}^{k}   b_{l_j}(r)\overline{b}_{l_{\s(j)}} = k!\Big(\sum_{k=1}^{n(r)} \eee^{2\s x_{k}(r)-(1+2\s^2)r}\Big)^{k}
.
\ee
Observe that $\sum_{k=1}^{n(r)} \eee^{2\s x_{k}(r)-(1+2\s^2)r}=\mathcal{M}_{2\s,0}(r)$ which converges almost surely to $\mathcal{M}_{2\s,0}$. This proves\eqv(con.101).

The case $k'<k$ follows similarly. Take an optimal matching of the first $k'$
particles. The other particles will not be matched. Take one $l_1$ that has not
been matched. Along its path, we can again find the first macroscopic piece on
which $m(\cdot)$ is constant. Applying Lemma~\ref{Lem.help}, we get that the
contribution is the largest if
$\max_{j\in{1,\dots,k',1,\dots,k}}d(l_1,l_{j})<R$, for $R$ large enough. Observe
that,
 \be\Eq(ny.31)
 \E\Big[\sum_{k=1}^{n(t)}\eee^{\s x_k(t) +i\tau z_k(t)-(\frac{1}{2}+\s)t} ~\Big|~ \mathcal{F}_R\Big]=\sum_{k=1}^{n(R)}\eee^{\s x_k(R) +i\tau z_k(R)-(\frac{1}{2}+\s)t} \eee^{(\s^2-\t^2+1-2\s^2+i2\t\s)(t-R)/2}
 .
\ee
Since in $B_3$ it holds that $1-\s^2-\t^2<0$, the summands the r.h.s.~of \eqv(ny.31) converge to zero as $t\uparrow\infty$.
This together with the argument in the even case implies Lemma \thv(Lem.const3).
\end{proof}

\subsection{Proof of Theorem \thv(clt.B3)}
\label{sec:clt.B3}

\begin{proof}[Proof of Theorem~\thv(clt.B3)]

By Lemma \thv(Lem.const3), conditionally on $\mathcal{F}_r$, the moments of
$N_{\s,\t}^{c,A}(t)$ converge to the moments of a $\mathcal{N}(0,C_{2,A}\MM_{2\s,0})$ a.s.\ as $t\uparrow \infty$ and then $r\uparrow
\infty$. Since the
normal distribution is uniquely characterised by its moments, this implies
convergence in distribution.  Moreover, by Lemma~\thv(Lem.const1) and
Lemma~\thv(Lem.const2),
\be \Eq(con.501)
\wlim_{A\uparrow\infty} \wlim_{t\uparrow \infty} \mathcal{L}\left[N_{\s,\t}(t)-N_{\s,\t}^{c,A}(t) \right] = \delta_0, 
\ee
and $\lim_{A\to\infty} C_{2,A}=C_2$. The claim of Theorem~\thv(clt.B3) follows.
 \end{proof}

\section{The boundaries}
\label{sec:boundaries}

In this section, we provide the proofs of Theorems~\thv(TH.B1.2), \thv(CLT.2)
and Proposition~\ref{CLT.B} describing the limiting fluctuations of the
partition function on all boundaries between the phases, i.e., on the 1D
manifolds $B_{1,2} = \overline{B_1} \cap \overline{B_2}$, $B_{1,3} =
\overline{B_1} \cap \overline{B_3}$, and $B_{2,3} = \overline{B_2} \cap
\overline{B_3}$.

\subsection{The boundary between phases $B_1$ and $B_3$}
\label{sec:B13}

\begin{proof}[Proof of Theorem \thv(CLT.2)]
The proof of Theorem \thv(CLT.2) works as in phase $B_3$. Observe first that
\be\Eq(clt.2.2)
\E\left[\frac{\XX_{\b,\rho}(t)}{ \sqrt{t} \eee^{t(1/2+\s^2)}}\right]=\frac{1}{\sqrt{t}}.
\ee
Moreover, let
\be \Eq(clt.2.3)
\hat{N}_{\s,\t}(t) := t^{-1/2}N_{\s,\t}(t) \quad \mbox{and} \quad   \hat{N}_{\s,\t}^{c,A}(t):=t^{-1/2} N_{\s,\t}^{c,A}(t).
\ee
By Lemma \thv(Thm.secondmoment) (ii),
\be \Eq(clt.2.4)
 \lim_{t\uparrow \infty}\E\left[\vert \hat{N}_{\s,\t}(t)\vert^2\right]=C_3.
\ee
Now, we need the following.
\begin{lemma}\TH(Lem.const3')
 For $\b\in B_3$,
\be\Eq(con.200)
\lim_{t\to\infty}\E\left[\vert \hat{N}_{\s,\t}^{c,A}(t)\vert^2\right]=C_{3,A}
,
\ee
with $\lim_{A \uparrow \infty } C_{3,A}=C_3$
and, for $k\in \N$, we have
\be \Eq(con.201)
\lim_{A\uparrow\infty }\lim_{r\uparrow \infty}\lim_{t\to \infty} \E\left[\vert \hat{N}_{\s,\t}^{c,A}(t)\vert^{2k} ~\big\vert~ \mathcal{F}_r\right]=k! (C_3\M_{2\s,0})^k \quad \mbox{a.s. and in }L^1.
\ee
Moreover, for $k'<k$,
\be\Eq(con.202)
\lim_{A\uparrow \infty}\lim_{r\uparrow \infty}\lim_{t\to \infty} \E\left[\hat{N}_{\s,\t}^{c,A}(t)^k\overline{ N_{\s,\t}^{c,A}(t)}^{k'}~\big\vert~ \mathcal{F}_r\right] = 0 \quad \mbox{a.s. and in }L^1.
\ee
 
\end{lemma}

\begin{proof}
The proof of Lemma~\thv(Lem.const3') is a rerun of the proof of Lemma~\thv(Lem.const3). 
\end{proof}

The claim of Theorem~\thv(CLT.2) follows with the very same arguments as the proof of Theorem~\thv(CLT.1).

\end{proof}

\subsection{Real critical point $\beta = \sqrt{2}$}

For $\vert \s\vert=1/\sqrt{2}$, the following scaling of the martingale $\mathcal{M}_{1,0}(t)$ plays an important role
\begin{equation}\Eq(bound.10)
\mathbb{M}^{\mathrm{SH}}_{1,0}(t) := \sqrt{t}\sum_{i=1}^{n(t)}\eee^{-\sqrt{2}(\sqrt{2}t-x_k(t))}
.
\end{equation}
$\mathbb{M}^{\mathrm{SH}}_{1,0}(t)$ is called \textit{critical additive martingale} and the rescaling appearing
in the r.h.s.\ of \eqref{bound.10} is referred to as \textit{Seneta-Heyde
scaling}. The limiting behaviour of $\mathbb{M}^{\mathrm{SH}}_{1,0}$ in the setting of branching
random walks has  been first analysed in \cite{AS14}. An alternative proof is
given in \cite{KM15}. As $t \to \infty$, \eqref{bound.10} converges in probability to a limiting random variable $\mathbb{M}^{\mathrm{SH}}_{1,0}$.

\begin{lemma}\TH(TH.real)
Denote $\M_{1,0}^{SH}(t):=
\sqrt{t}\sum_{i=1}^{n(t)}\eee^{-\sqrt{2}(\sqrt{2}-x_k(t))}$ and $\M_{1,0}^{SH}
:= \left(\frac{2}{\pi}\right)^{1/2}\mathcal{Z}$, where $\mathcal{Z}$ is the
limit of the derivative martingale, cf.~\eqref{real.2}. Then, for
$\beta=\sqrt{2}$, the following convergence holds in probability 
\be\Eq(real.1ak)
\M_{1,0}^{\mathrm{SH}}(t) \underset{t \to \infty}{\overset{\P}{\longrightarrow}}
\M_{1,0}^{\mathrm{SH}}
. 
\ee
\end{lemma}
\begin{proof}
The proof is just an adaptation of the result for the branching random walk
(see \cite[Section~6.5]{KM15}).
\end{proof}

\subsection{The boundary between phases $B_2$ and $B_3$; and the triple point $\beta = (1+i)/\sqrt{2}$}
\label{sec:B_2-B_3-boundary}

In this section, we prove the convergence of the moments of the rescaled
partition function on the boundary between phases $B_2$ and $B_3$ to the moments
of a Gaussian random variable with random variance in probability  which is the
content of Theorem~\ref{CLT.B}.

\begin{proof}
[Proof of Theorem \thv(CLT.B)] \paragraph{(i)} The proof of Theorem~\thv(CLT.B)
(i) is a modification of the proof of Theorem \thv(CLT.1) (ii) in the following
way.

\begin{lemma}\TH(Lem.const3.3)
 For $\b$ with $\s=\frac{1}{\sqrt{2}}$, $\rho\in[-1,1]$ and binary branching
\be\Eq(con.300)
\lim_{t\to\infty}\E\left[\left\vert N_{\s,\t}^{c,A}(t)\right\vert^2\right]=C_{2,A}
,
\ee
and, for $k\in \N$, we have
\be \Eq(con.301)
\lim_{r\uparrow \infty}\lim_{t\to \infty} r^{\frac{2k}{4}}\E\left[\left\vert N_{\s,\t}^{c,A}(t)\right\vert^{2k} ~\vert~ \mathcal{F}_r\right]=k! (C_{2,A}\M^{\mathrm{SH}}_{1,0})^k \quad \text{ in probability,}
\ee
where $\M^{\mathrm{SH}}_{1,0}$ is the martingale defined in \eqref{real.1ak}.
Moreover, for $k'<k$,
\be\Eq(con.302)
\lim_{r\uparrow \infty}\lim_{t\to \infty} \E\left[N_{\s,\t}^{c,A}(t)^k\overline{ N_{\s,\t}^{c,A}(t)}^{k'}~\big\vert~ \mathcal{F}_r\right] =0\quad \mbox{in probability. }
\ee

\end{lemma}
\begin{proof}
The proof is a rerun of the proof of Lemma~\thv(Lem.const3) with the only
difference that the martingale $\M^{\mathrm{SH}}_{1,0}$ only converges in probability
towards $\M^{\mathrm{SH}}_{1,0}$ as $t\uparrow \infty$ and that there an additional
factor $r^{1/4}$ needed.
\end{proof}

Since \eqv(con.301) and \eqv(con.302) only hold in probability, using the same
method as in the proof of Theorem~\thv(CLT.1), we get the corresponding weak
convergence result. \paragraph{(ii)} For the triple point, the argument is
similar to \paragraph{(i)} but with the moments as given in
Lemma~\thv(Lem.const3') with $\M_{2\s,0}$ replaced by $\M^{\mathrm{SH}}_{1,0}$.
\end{proof}

\subsection{The boundary between phases $B_1$ and $B_2$}
\label{sec:B12}

In this section, we prove Theorem~\thv(TH.B1.2).
\begin{proof}[Proof of Theorem~\thv(TH.B1.2)]

For $\b\in \bar{B}_1\cap\bar{B}_2
\setminus\{\b=\sqrt{2},\b=\frac{1}{\sqrt{2}}(1+i)\}$, consider in the same way
as in the proof of Theorem \thv(TH.B1) the $\frac{\sqrt{2}}{\gamma}$-moment for
some $\gamma>\s$ and $\sqrt{2}\gamma>1$. Then, a rerun of the computation
starting from \eqv(m.1) up to \eqv(m.9) bounds the
$\frac{\sqrt{2}}{\gamma}$-moment from above by
     \be\Eq(m.10) 
      \begin{aligned}
      & \int_0^t \dd q~\eee^{\frac{(\s^2- \t^2+2)(t-q)-2t-(\s^2-\t^2)t}{\sqrt{2}\s} } \eee^{\frac {\gamma^2}{\s^2}q +q} 
      \\
      &= \int_0^t \dd q~\eee^{\frac{\left(\t^2-(\s-\sqrt{2})^2+\left(\frac {\gamma^2}{\s^2}-1\right)\right)q}{\sqrt{2}\s}}
      \\
      &= \int_0^t \dd q~\eee^{\frac{\left(\frac{\gamma^2}{\s^2}-1\right)q}{\sqrt{2}\s}} 
      ,
     \end{aligned}
     \ee
since $|\t|+|\s|=\sqrt{2}$. The r.h.s.~of~\eqv(m.10) is uniformly bounded by a
constant. Hence, $\MM_{\s,\t}(t)$ is in $L^p$ for some $p>1$. Hence, it
converges a.s.\ and in $L^1$. The limit is non-degenerate because
$\E\left[\MM_{\s,\t}(t)\right]=1$ and Theorem~\thv(TH.B1.2) follows.
\end{proof}

\section{Proof of Theorem~\ref{Cor:phase-diagram}}
\label{sec:phase-diagram}

In this section, as a consequence of the fluctuation results of the previous sections, we
derive the phase diagram shown on Fig.~\ref{fig-rem-phase-diagram}.

\begin{proof}[Proof of Theorem~\ref{Cor:phase-diagram}]
Convergence in probability for $\beta \in  B_1$ and $ B_3$   in
\eqref{eq:limiting-log-partition-function} follows from Theorems~\ref{TH.B1} and
\ref{CLT.1} (ii) by \cite[Lemma~3.9~(1)]{KaKli14}. Convergence for the glassy
phase $\beta \in \overline{B_2}$ was shown in \cite{HK15}. For the boundaries
between all three phases, the formula \eqref{eq:limiting-log-partition-function}
follows from the continuity of the limiting log-partition function.
\end{proof}

\bibliographystyle{plain}

\end{document}